\documentclass[11pt]{amsart}

\usepackage{amsmath}
\usepackage{amssymb}
\usepackage{graphicx}
\usepackage{amsthm}

\newtheorem{theorem}{Theorem}

\newtheorem{mclaim}[theorem]{Main Claim}

\newtheorem{lemma}[theorem]{Lemma}

\newtheorem{obs}[theorem]{Observation}

\theoremstyle{definition} 
\newtheorem{definition}[theorem]{Definition}

\theoremstyle{remark}
\newtheorem{myremark}[theorem]{Remark}

\newtheorem{conclusion}[theorem]{Conclusion}

\newtheorem{notation}[theorem]{Notation}

\newcommand{\rest}{{\upharpoonright}}

\newcommand{\conc}{{}^\frown\!}
\newcommand{\bool}{{\bf B}}

\newcommand{\rlg}{{\rm lg}}

\newcommand{\rds}{{\rm ds}}
\newcommand{\rrk}{{\rm rk}}

\newcommand{\st}{{such that}}
\newcommand{\sequ}{{sequence}}

\newcommand{\incr}{{increasing}}

\newcommand{\Wlog}{{Without loss of generality}}
\newcommand{\wolog}{{without loss of generality}}

\newcommand{\dec}{{decreasing}}

\newcommand{\then}{{\underline{then}}}

\newcommand{\uif}{{\underline{if}}}
\newcommand{\uIf}{{\underline{If}}}
\newcommand{\uiff}{{\underline{iff}}}

\newcommand{\bb}{\mathbb}
\newcommand{\mat}{\mathcal}
\newcommand{\structA}{\mathfrak A}
\newcommand{\structB}{\mathfrak B}

\newcommand{\ordone}{<^1_{\ell x}}
\newcommand{\ordtwo}{<^2_{\ell x}}
\newcommand{\order}{<^*_{\ell x}}

\newcommand{\seq}[1]{\langle #1 \rangle}
\newcommand{\set}[1]{\{ #1 \}}
\newcommand{\size}[1]{| #1 |}
\newcommand{\logic}[1]{{\bb L_{{#1},{#1}}}}

\title{ Uniforming $n$-place functions on $T\subseteq {\rm ds}(\alpha)$ }
\author{Esther Gruenhut}
\address{Institute of Mathematics
 The Hebrew University of Jerusalem,
 Jerusalem 91904, Israel.
 {\tt\ e-mail: esth@math.huji.ac.il}}
\author{Saharon Shelah}
\address{Institute of Mathematics
 The Hebrew University of Jerusalem,
 Jerusalem 91904, Israel
 and  Department of Mathematics
 Rutgers University
 New Brunswick, NJ 08854, USA
 {\tt\ e-mail: shelah@math.huji.ac.il}}
\thanks{Research supported by the United States-Israel Binational
Science Foundation (Grant no. 2002323). Publication 909 in
Shelah's archive.
\\ AMS 200 classification 03E05, 05C15.
\\ \sl{Key words: Set Theory, partition relation, well founded trees, scattered order types}}

\begin{document}

\begin{abstract}
In this paper the Erd\H os-Rado theorem is generalized to the class of well founded trees.
We define an equivalence relation on the class $\rds(\infty)^{<\aleph_0}$ ( finite sequences of decreasing sequences of ordinals) with $\aleph_0$ equivalence classes, and for $n<\omega$ a notion of $n$-end-uniformity for  a colouring  of $\rds(\infty)^{<\aleph_0}$ with $\mu$ colours. We then  show that for every ordinal $\alpha$, $n<\omega$ and cardinal $\mu$ there is an ordinal $\lambda$ so that for any colouring $c$ of $T=\rds(\lambda)^{<\aleph_0}$ with $\mu$ colours,  $T$ contains $S$ isomorphic to $\rds(\alpha)$ so that $c\rest S^{<\aleph_0}$ is $n$-end uniform. For $c$ with domain $T^n$ this is equivalent to finding $S\subseteq T$ isomorphic to $\rds(\alpha)$ so that $c\rest S^{n}$ depends only on the equivalence class of the defined relation, so in particular $T\rightarrow(\rds(\alpha)) ^n_{\mu,\aleph_0}$. We also draw a conclusion on colourings of $n$-tuples from a scattered linear order.
\end{abstract}

\maketitle
This paper is a natural continuation of  \cite{796} in which Shelah and Komj\'ath prove that for any scattered order type $\varphi$ and cardinal $\mu$ there exists a scattered order type $\psi$ \st\  $\psi\rightarrow[\varphi]^n_{\mu,\aleph_0}$. This was proved by a theorem on colourings of well founded trees. By Hausdorff's characterization (see \cite{haus} and \cite{rosen} ) every scattered order type can be embedded in a well founded tree, so we can deduce  a  natural generalization of their theorem to the $n$-ary case, i.e for every scattered order type $\varphi$, $n<\omega$, and cardinal $\mu$ there is a scattered order type $\psi$ \st\ $\psi\rightarrow(\varphi)^n_{\mu,\aleph_0}$.
\par\noindent

We start with a few definitions.
\begin{definition} \label{ds}
For an ordinal $\alpha$ we define $\rds (\alpha)=\{\eta:\eta$ a \dec\ \sequ\ of ordinals
$<\alpha\}$. By $\rds(\infty)$ we mean the class of decreasing sequences of ordinals.
\end{definition}
 We say $T \subseteq \rds(\infty)$ is a tree when $T$ is  non-empty and  closed under initial segments. $T,S$ will denote trees. For $S\subseteq T\subseteq\rds(\infty)$ we say that $S$ is a subtree of $T$ if it is also a tree. We use the following notation:
\begin{notation}
 \begin{enumerate}
 \item For $\eta,\nu\in\rds(\infty)$ by $\eta\cap\nu$ we mean $\eta\rest\ell$ where $\ell$ is  maximal \st\  $\eta\rest\ell=\nu\rest\ell$.
 \item For $\eta\in \rds(\infty)$ and a tree $T\subset\rds(\infty)$ we define $$\eta\conc T=\set{\rho:\rho\trianglelefteq\eta \vee(\exists\nu\in T)(\rho=\eta\conc\nu)}$$
 \end{enumerate}
\end{notation}

Note that for $\eta\in\rds(\infty\backslash\set{\seq{}})$ and $\set{\seq{}}\subsetneq T\subseteq\rds(\infty)$ if $\eta(\rlg(\eta)-1)>\sup\set{\rho(0):\rho\in T}$ then $\eta\conc T\subseteq\rds(\infty)$.
\begin{definition} We define the following four binary relations on $\rds(\infty)$:
\label{ord}
\begin{enumerate}
\item Let $\ordone$ be the two place relation
on $\rds (\infty)$ defined by $\eta \ordone \nu$ iff  one of the following: $(\exists \ell)(\eta(\ell)<\nu (\ell)$ or $\eta \rest\ell=\nu \rest \ell)$
\underline{or} $\eta\triangleleft \nu$.
\item Let $\ordtwo$ be the two place relation on $\rds (\infty)$ defined by
$\eta \ordtwo \nu$ iff one of the following:
$(\exists \ell) (\eta (\ell)<\nu (\ell) $ or $\eta \rest \ell=\nu \rest \ell) \hbox{ \underline{or} }
\nu \triangleleft \eta.$
\item $<^*_{\ell x}=\ordone \cap \ordtwo$.
\item Let $<^3$ be the  two place relation on $\rds(\infty)$  defined by $\eta<^3\nu$ iff one of the following:  $\eta\triangleleft\nu$ or  for the maximal $\ell$ \st\ $\eta\rest\ell=\nu\rest\ell$ if $\ell$ is even then $\eta(\ell)<\nu(\ell)$ and if $\ell$ is odd then $\eta(\ell)>\nu(\ell)$.
\end{enumerate} \end{definition}
It is easily verified that   $\ordone,\ordtwo$ and $<^3$ are complete
orders of $\rds(\infty)$, and therefore $\order$ is a partial order.
The following remark refers to their order types defined by
$\ordone,\ordtwo$ and $<^3$ on $\rds(\infty)$ or $\rds(\alpha)$.

\begin{obs}
\label{obs-ord}
\begin{enumerate}
\item $\ordone,\ordtwo$ are well orderings for $\rds(\infty)$.
\item $(\rds(\alpha),<^3)$ is a scattered linear order type for every ordinal $\alpha$.
\item Every scattered linear order type can be embedded in $(\rds(\alpha),<^3)$ for some ordinal $\alpha$.
\end{enumerate}
\end{obs}
\begin{proof}
\begin{enumerate}
\item Let $\emptyset\neq A\subseteq\rds(\infty)$, we define by induction on $n<\omega$ an element $a_n$ in the following manner $a_0=\min\set{\eta(0):\eta\in A}$, assume $a_0,\cdots,a_{n-1}$ have been chosen so that $\seq{a_k:k<n}\in\rds(\infty)$ and  for every $\eta\in A$ $\seq{a_k:k<n}\leq^2_{\ell x}\eta\rest{n}$ (if $\rlg(\eta)\leq n$ then $\eta\rest{n}=\eta$). Now choose $a_{n}=\min\set{\eta(n):\eta\in A\wedge \eta\rest n=\seq{a_k:k<n}}$, if that set isn't empty.  As the \sequ\ derived in the above manner is a decreasing sequence of ordinals it is finite, say $a_0,\cdots a_{n-1}$ have been defined and $a_{n}$ cannot  be defined, we will show that $\bar a =\seq{a_k:k<n}$ is the minimal element of $A$ with respect to $\ordtwo$. By the definition of the sequence there is an  $\eta\in A$ so that $\eta\rest n=\bar a$, if $\rlg(\eta)>n$ then we could have defined $a_n$, so $\eta=\bar a$ and in particular $\bar a\in A$, and for every $\eta\in A\backslash\set{\bar a}$ we have $\bar a\ordtwo\eta$. Let $n_*=\min\set{m:\bar a \rest m\in A}$ so $\bar a\rest n_*$ is the minimal $\ordone$ element in $A$.
\item The proof is by induction on $\alpha$. Assume that $(\rds(\beta),<^3)$ is a scattered linear order type for every $\beta<\alpha$, and assume towards contradiction that $\bb Q$ can be embedded in $(\rds(\alpha),<^3)$, $q\mapsto \eta_q$. Let $C=\set{\ell:(\exists p,q\in\bb Q )(\eta_p(\ell)\neq\eta_q(\ell))}$, $\ell=\min C$ and $\Gamma=\set{\beta:(\exists q\in\bb Q)(\eta_q(\ell)=\beta)}$. \Wlog\ $\ell$ is even and for $\beta_0=\min\Gamma$, $\beta_1=\min\Gamma\backslash\set{\beta_0}$  there are $q_0<q_1\in \bb Q$ so that $\eta_{q_i}(\ell)=\beta_i$, $i=0,1$. Now $(q_0,q_1)=B_0\cup B_1$ where $B_i=\set{p\in(q_0,q_1):\eta_p(\ell)=\beta_i}$. For some $i\in\set{0,1}$ the set $B_i$ contains an interval of $\bb Q$ and is embedded in  $(\eta_{q_i}\rest(\ell+1)\conc\rds(\beta_i),<^3)$ but this would imply that $\bb Q$ can be embedded in $(\rds(\beta_i),<^3)$ which is a contradiction to the induction hypothesis.
\item By Hausdorff's characterization it is enough to show  for ordinals $\alpha$ and  $\beta$ that both $A_{\alpha,\beta}=(\rds(\alpha),<^3)\times\beta$ and   $A_{\alpha,\beta^*}=(\rds(\alpha),<^3)\times\beta^*$ can be embedded in $(\rds(\alpha+\beta\cdot2+1),<^3)$. The embedding is given as follows, for $(\eta,\gamma)\in A_{\alpha,\beta}$ we have $(\eta,\gamma)\mapsto \seq{\alpha+\beta+\gamma+1,\alpha+\beta}\conc\eta$, and for $(\eta,\gamma)\in A_{\alpha,\beta^*}$ we have $(\eta,\gamma)\mapsto \seq{\alpha+\beta\cdot2,\alpha+\beta+\gamma}\conc\eta$.
\end{enumerate}
\end{proof}

\begin{definition}
\label{1.3}
 For trees $T_1,T_2\subset\rds(\infty)$, $f:T_1\to T_2$ is an embedding of $T_1$ into $T_2$ if $f$ preserves level, $\triangleleft$ and  $\ordone$ (or equivalently, $\ordtwo,\order$ or $<^3$).
\end{definition}
\begin{obs}
\label{obs-embedding}
For trees $T_1,T_2\subset\rds(\infty)$, if $f:T_1\to T_2$ preserves level and $\triangleleft$ then in order to determine whether $f$ is an embedding it is enough to check for $\eta\in T_1$ and ordinals $\gamma_1<\gamma_2$ \st\ $\nu_i=\eta\conc\seq{\gamma_i}\in T_1$ ($i=1,2$) that $f(\nu_1)\order f(\nu_2)$.
\end{obs}
As $T\subseteq\rds(\infty)$ is well founded, i.e there are no infinite branches, it is natural to define a rank function. in the following definition $\rrk_{T,\mu}$ isn't the standard rank function but for $\mu=1$ we get a similar definition to the usual definition of a rank on a well founded tree.
\begin{definition}
For a tree $T\subset\rds(\infty)$ and cardinal $\mu$ define ${\rrk}_{T,\mu} (\eta):\rds(\infty)\to\set{-1}\cup\mathrm{Ord}$  by induction on $\alpha $ as follows:
\label{rank}
\begin{enumerate}
\item[(a)] ${\rm rk}_{T,\mu} (\eta) \geq 0$ iff $\eta \in T$.
\item[(b)] ${\rm rk}_{T,\mu} (\eta) \geq \alpha+1$ \uiff\
$\mu\leq \size{\{\gamma: \eta\conc\langle\gamma\rangle\in T \wedge
{\rm rk}_{T,\mu}(\eta\conc\langle\gamma\rangle)\geq\alpha\}}
$.
  \item[(c)] ${\rm rk}_{T,\mu} (\eta) \geq \delta$ limit \uiff\
$(\forall \alpha<\delta)( {\rm rk}_{T,\mu} (\eta) \geq \alpha)$.
\end{enumerate}
We say that $\rrk_{T,\mu}(\eta)=\alpha$ iff $\rrk_{T,\mu}(\eta)\geq\alpha$ but
$\rrk_{T,\mu}(\eta)\ngeq\alpha+1$.
\\ Denote $\rrk_{T,\mu}(T)=\rrk_{T,\mu}(\seq{})$, and $\rrk_T(\eta)=\rrk_{T,1}(\eta)$.
\end{definition}
\begin{definition}\label{reduced-rank}
 For a tree $T\subset\rds(\infty)$, $\eta\in T$ and cardinals $\mu,\lambda$ we define the reduced
 rank ${\rrk}_{T,\mu}^\lambda (\eta)=\min\set{\lambda,\rrk_{T,\mu}(\eta)}$.
\end{definition}
We first note a few properties of the rank function.
\begin{obs}
\label{obs-rank} For $\eta\in T\subset\rds(\infty)$ and an ordinal $\alpha$ we have:
\begin{enumerate}
\item  For cardinals $\mu\leq\mu'$ we have
    $\rrk_{T,\mu}(\eta)\geq\rrk_{T,\mu'}(\eta)$, and in
    particular $\rrk_{T}(\eta)\geq\rrk_{T,\mu}(\eta)$
\item $\rrk_T(\eta)=\cup\set{\rrk_T(\eta\conc\seq{\gamma})+1:\eta\conc\seq{\gamma}\in T}.$
\item $\rrk_{\rds(\alpha)}(\seq{})=\alpha$.
\item \uIf\ ${\rm rk}_{T,\mu} (\eta)\geq \alpha$,  $\mu\geq\alpha$
\then\  we can embed $\eta\conc\rds(\alpha)$ into $T$, so that $\rho\mapsto\rho$ for $\rho\trianglelefteq\eta$.
\end{enumerate}
\end{obs}
\begin{proof}
\begin{enumerate}
 \item[$3$] The proof is by induction on $\alpha$.\\ For
     $\alpha=0$ this is obvious. Assume correctness for every
     $\beta<\alpha$.
     $\rds(\alpha)=\bigcup\limits_{\beta<\alpha}\set{\seq{\beta}\conc\nu:\nu\in\rds(\beta)}$.
     For every $\beta<\alpha, \nu\in\rds(\beta)$ we have
     $\rrk_{\rds(\alpha)}(\seq{\beta}\conc\nu)=\rrk_{\rds(\beta)}(\nu)$, therefore (the last
     equality is due to the induction hypothesis):\\ $\begin{array}{lcl} \cup\set{\rrk_{\rds(\alpha)}
     (\seq{\beta}\conc\nu)+1:\nu\in\rds(\beta)}&=&\cup\set{\rrk_{\rds(\beta)}
     (\nu)+1:\nu\in\rds(\beta)}\\ {}&=&\rrk(\rds(\beta))\\
     {}&=&\beta\end{array}$\\
We therefore have $\rrk(\rds(\alpha))=\cup\set{\beta+1:\beta<\alpha}=\alpha$
\item[$4$] The proof is by induction on $\alpha$.
\\For $\alpha=0$ there is nothing to prove.
\\ Assume correctness for every $\beta<\alpha$, and $\rrk_{T,\mu} (\eta)\geq \alpha$, $\alpha\leq\mu$.
For $\beta<\alpha$ let $C_\beta=\set{\gamma:\rrk_{T,\mu}(\eta\conc\seq{\gamma})\geq\beta}$, so $\size{C_\beta}\geq\mu$ and  $C_\beta\subseteq C_{\beta'}$ for $\beta'<\beta<\alpha$.  By induction on $\beta<\alpha$ we can choose an \incr\ \sequ\ of ordinals $\gamma_\beta$ \st\ $\gamma_\beta=\min \Gamma_\beta$ where $\Gamma_\beta=\set{\gamma\in C_\beta:(\forall\beta'<\beta)(\gamma>\gamma_{\beta'})}$. Assume towards contradiction that $\Gamma_\beta$ is empty, and let $C_\beta'=\seq{\gamma_{\beta'}:\beta'<\beta}\cap C_\beta$. For every $\gamma\in C_\beta\backslash C_\beta'$ (and there is such $\gamma$ as $\size{C_\beta}\geq\mu$ whereas $\size{C_\beta'}\leq\size{\beta}<\mu$) as $\gamma\notin\Gamma_\beta$ then there is $\beta'<\beta$ \st\ $\gamma<\gamma_{\beta'}$, assume $\beta'$ is minimal with this property, but that contradicts the choice of $\gamma_{\beta'}$.
\\By the induction hypothesis for every $\beta<\alpha$ there is $\varphi_\beta$ which embeds $(\eta\conc\seq{\gamma_\beta})\conc\rds(\beta)$ in $T$ so that $\varphi_\beta\rest\set{\rho:\rho\trianglelefteq\eta\conc\seq{\gamma_\beta}}=Id$. We now define $\varphi_\alpha:\eta\conc\rds(\alpha)\to T$ in the following manner, if $\rho\trianglelefteq\eta$ then $\varphi_{\alpha}(\rho)=\rho$, else $\rho=\eta\conc\nu$ for some $\nu\in\rds(\alpha)$, so there is $\beta<\alpha$ \st\ $\nu=\seq{\beta}\conc\nu_1$ with $\nu_1\in\rds(\beta)$, and we define $$\varphi_\alpha(\rho)= \varphi_\beta(\eta\conc\seq{\gamma_\beta}\conc\nu_1).$$
     $\varphi_\alpha$ obviously preserves level.
    \\For $\rho_1\triangleleft\rho_2$ in $\eta\conc\rds(\alpha)$ if $\rho_1\trianglelefteq\eta$ then obviously $\varphi_\alpha(\rho_1)\triangleleft\varphi_\alpha(\rho_2)$, and otherwise for some $\beta<\alpha$  we have $\rho_i=\eta\conc\seq{\beta}\conc\nu_i$, $i\in\set{1,2}$, $\nu_1\triangleleft\nu_2\in\rds(\beta)$, and  as $\varphi_\beta$ is an embedding  we have:
     $$
     \varphi_\alpha(\rho_1)=\varphi_\beta(\eta\conc\seq{\gamma_\beta}\conc\nu_1)\triangleleft \varphi_\beta(\eta\conc\seq{\gamma_\beta}\conc\nu_2)=\varphi_\alpha(\rho_2).
     $$
     For $\rho\in\eta\conc\rds(\alpha)$, $\gamma_1<\gamma_2$ ordinals \st\ for $i=1,2$ $\rho_i=\rho\conc\seq{\gamma_i}\in \eta\conc\rds(\alpha) $, necessarily $\eta\trianglelefteq\rho$ and there are $\beta_1\leq\beta_2<\alpha$, $\nu_i\in\rds(\beta_i)$ so that $\rho_i=\eta\conc\seq{\beta_i}\conc\nu_i$. If $\beta_1=\beta_2=\beta$ then $\nu_1\order\nu_2$, and as $\varphi_\beta$ is an embedding, $$\varphi_\alpha(\rho_1)=\varphi_\beta(\eta\conc\seq{\gamma_\beta}\conc\nu_1)\order \varphi_\beta(\eta\conc\seq{\gamma_\beta}\conc\nu_2)=\varphi_\alpha(\rho_2)
     $$
     On the other hand, if $\beta_1\neq\beta_2$ then $\varphi_\alpha(\rho_i)(\rlg(\eta))=\gamma_{\beta_i}$, and as $\gamma_{\beta_1}<\gamma_{\beta_2}$, also in this case $\varphi_\alpha(\rho_1)\order \varphi_\alpha(\rho_2)$.
     \\By Observation \ref{obs-embedding} $\varphi_\alpha$ is an embedding, and by definition $\varphi_\alpha\rest\set{\rho:\rho\trianglelefteq\eta}=Id$.
\end{enumerate}
\end{proof}

\medskip\medskip\noindent
The following theorem was was proved By Komj\'{a}th and Shelah in \cite{796}:
\begin{theorem}
\label{796}
Assume $\alpha$ is an ordinal and $\mu$ a cardinal. Set $\lambda=(\size{\alpha}^{\mu^{\aleph_0}})^+$, and let $F:\rm ds(\lambda^+)\rightarrow\mu$. Then there is an embedding $\varphi:\rds(\alpha)\to\rds(\lambda^+)$ and a function $c:\omega\rightarrow\mu$ \st\ for every $\eta\in\rm ds(\alpha)$ of length $n+1$
$$
F(\varphi(\eta))=c(n).
$$
\end{theorem}
In what follows we will generalize the above theorem, in the process we will use infinitary logics.
For the readers convenience we include the following definitions.
\begin{definition}
\label{logics-def}
\begin{enumerate}
\item
For infinite cardinals $\kappa,\lambda$, and a vocabulary $\tau$ consisting of a list of relation
 and function symbols and their `arity' which is finite,  the infinitary language
 $\bb L_{\kappa,\lambda}$ for $\tau$  is defined in a similar manner to first order logic.
  The first subscript, $\kappa$, indicates that formulas have $<\kappa$ free variables and that we can join together $<\kappa$ formulas by
  $\bigwedge$ or $\bigvee$, the second subscript, $\lambda$, indicates that we can put $<\lambda$
  quantifiers together in a row.
\item Given a structure $\structB$ for $\tau$ we say that $\structA$ is  an
$\bb L_{\kappa,\lambda}$-elementary submodel (or substructure), and denote
$\structA\prec_{\kappa,\lambda}\structB$ or $\structA\prec_{\bb L_{\kappa,\lambda}}\structB$,
if $\structA$ is a substructure of $\structB$ in the regular manner,
and for any $\bb L_{\kappa,\lambda}$ formula $\varphi$ with $\gamma$ free variables and $\bar a\in{}^\gamma\size{\structA}$ we have
$$\structB\models\varphi(\bar  a)\Leftrightarrow \structA\models\varphi(\bar a).$$
The Tarski-Vought condition for  a substructure $\structA$ of $\structB$ to be an elementary submodel is that  for any
$\bb L_{\kappa,\lambda}$-formula $\varphi$ with parameters $\bar a\subseteq\structA$
we have
$$\structB\models\exists\bar x\varphi(\bar x\bar a)\Rightarrow \structA\models\exists\bar x\varphi(\bar x\bar a).$$
\item A set $X$ is transitive if for every $x\in X$ we have $x\subseteq X$.
\item For every set $X$ there exists a minimal transitive set,
which is denoted by $TC(X)$, such that $X\subseteq TC(X)$.
 \item  For an infinite regular cardinal $\kappa$ we define $$\mathcal
 H(\kappa)=\set{X:\size{TC(X)}<\kappa}.$$
\end{enumerate}
\end{definition}
\begin{myremark}
\label{logics-use}
In this paper the main use of infinitary logic will be in the following manner:
\begin{enumerate}
\item  $\tau$ will consist of the two binary relations $\in$ and $<^*$, so $\size{\logic{\kappa^+}(\tau)}=2^\kappa$.
\item If $\kappa'\leq\kappa, \lambda'\leq\lambda$ and $\structA\prec_{\kappa,\lambda}\structB$ then also $\structA\prec_{\kappa',\lambda'}\structB$.
\item $\prec_{\kappa,\lambda}$ is a transitive relation.
\item For an infinite cardinal $\mu$ let $\kappa=\mu^+$, $\lambda=2^\mu$, so $\kappa$ is regular and $\lambda^{<\kappa}=\lambda$. Recall that for a structure $\structB$ and $X\subseteq\|\structB\|$ \st\ $\size{X}+\tau\leq\lambda\leq\structB$ there is an elementary $\logic{\kappa}$ submodel $\structA$ of $\structB$ of cardinality $\lambda$ which includes $X$.
    \\For  further reference on this point see \cite{Di}.
\item If $\structA\prec_{\kappa,\kappa}\structB$ and $x$ is definable in $\structB$ over $\structA$ (i.e with parameters in $\structA$) by an $\logic{\kappa}$-formula, then it is also definable in $\structA$ by the same formula. In particular if $\structA\prec_{\kappa,\kappa}\structB$ and $X\subseteq\size{\structA}, \size{X}<\kappa$ then $X\in|\structA|.$
\end{enumerate}
\end{myremark}

\begin{definition}\label{similarity}
We say two finite \sequ\ $\seq{
\eta_\ell:\ell<n}, \seq{ \nu_\ell:\ell<n}$ are
\underline{similar} when:
\begin{enumerate}
\item[(a)] $\rlg(\eta_\ell)=\rlg(\nu_\ell)$
for $\ell<n$.
\item[(b)] $\rlg(\eta_\ell \cap \eta_m)=
\rlg(\nu_\ell \cap \nu_m)$ for $\ell,m<n$.
\item[(c)] $(\eta_\ell \ordtwo\eta_m) \equiv(\nu_\ell\ordtwo \nu_m)$
for $\ell,m<n$ (equivalently, we could use $\ordone$).
\end{enumerate}
\end{definition}
\begin{obs}
\label{obs-sim}
\begin{enumerate}
\item Similarity is an equivalence relation and the number of
equivalence classes of finite sequences is $\aleph_0$.
 \item $\langle \eta_1,\ldots, \eta_k,\nu'\rangle$, $\langle
\eta_1,\ldots,\eta_k,\nu''\rangle$ are similar if
\begin{enumerate}
\item $\eta_1 \ordtwo \eta_2 \ordtwo \ldots
\ordtwo \eta_k$
 \item $\eta_k \ordtwo \nu'$
 \item $\eta_k \ordtwo \nu''$
 \item $\rlg(\nu')=\rlg(\nu'')$
 \item $\rlg(\nu' \cap \eta_k)= \rlg(\nu''\cap \eta_k)$
\end{enumerate}
\end{enumerate}
\end{obs}
\begin{proof}
\begin{enumerate}
 \item Similarity is obviously an equivalence relation.
\\The equivalence class of a finite sequence of $\rm ds(\infty)$ is determined by its length $n$, the lengths $\seq{n_i:i<n} $ of its elements, the lengths $\seq{n_{i,j}:i,j<n}$ of their intersections, and a permutation of $n$ (the order of the elements according to $\ordone$). Therefore for each $n<\omega$ there are $\aleph_0$ equivalence classes of sequences of length $n$, and so the number of equivalence classes of finite sequences of $\rm ds(\infty)$ is $\aleph_0$.
\item We need to show that
 $\rlg(\nu'\cap\eta_i)=\rlg(\nu''\cap\eta_i)$ for every $0<i<k$.
 \\ $\eta_k\ordtwo\nu'$ and  $\eta_k\ordtwo\nu''$. If $\nu'\triangleleft\eta_k$ then we also have $\rlg (\nu'' \cap \eta_k)= \rlg (\nu'\cap\eta_k) = \rlg(\nu')
=\rlg(\nu'')$ so $\nu''\triangleleft\eta_k$, and $\nu'=\nu''$. In this case obviously the required sequences are similar, so we can assume that there is $\ell$ \st\
$\eta_k\rest\ell=\nu'\rest\ell$ and
$\nu'(\ell)>\eta_k(\ell)$. By the
same reasoning as above we deduce that $\eta_k\rest\ell=\nu''\rest\ell$ and
$\nu''(\ell)\neq\eta_k(\ell)$ so necessarily
$\nu''(\ell)>\eta_k(\ell)$.
\end{enumerate}
\end{proof}
The last term we will need before moving on to the main theorem
is that of uniformity.
\begin{definition}
\label{uniformity} Let $T\subseteq \rds(\infty)$ be a tree, $c:[T]^{<\aleph_0}\rightarrow C$. We identify $u\in[T]^{<\aleph_0}$ with the $\ordtwo$-\incr\ \sequ\ listing it.
\begin{enumerate}
\item We say $T$ is $c$-uniform if for any similar
$u_1,u_2$ in $[T]^{<\aleph_0}$ we have
$c(u_1)=c(u_2)$.
\item We say $T$ is $c$-end-uniform
(or end-uniform for $c$) when
\\ \uif\  $\eta_1 \ordtwo \eta_2
\ordtwo \ldots
\ordtwo \eta_k \ordtwo \rho',
\rho''$ are in $T$ and
$\rlg (\rho')= \rlg (\rho''), \rlg
(\eta_k \cap \rho')=\rlg (\eta_k\cap
\rho'')$
(equivalently $\langle \eta_1\ldots
\eta_k,\rho'\rangle,
\langle \eta_1\ldots \eta_k, \rho''\rangle$
are similar-see \ref{obs-ord}(3))
\\ \then\ $c(\seq{\eta_1\ldots \eta_k,\rho'})=
c(\seq{\eta_1,\ldots,\eta_k,\rho''})$.
\item We say $T$ is $c$-$n$-end-uniform
(or $n$-end-uniform for c) when for $k<\omega$, $\eta_i,\rho'_j,\rho''_j\in\rds(\infty)$ $(0<i\leq k,0<j\leq n)$ \st
$$
\eta_1 \ordtwo<\eta_2 \ordtwo
\ldots \ordtwo \eta_k
\ordtwo \rho'_1\ordtwo\ldots \ordtwo
\rho'_n
$$
$$
\eta_1 \ordtwo < \eta_2 \ordtwo\ldots
\ordtwo \eta_k \ordtwo
\rho''_1 \ordtwo <\ldots< \rho''_n
$$
\uif\ those two sequences are similar \then
$$
c(\seq{\eta_1\ldots,\rho'_1\ldots})=
c(\seq{\eta_1\ldots \rho''_1\ldots}).
$$
\end{enumerate}
\end{definition}

We are now ready for the main theorem of this paper.
\begin{mclaim}
\label{main claim}
Given a tree $S\subseteq \rds (\infty)$ and a cardinal
$\mu$ we can find a tree $T \subseteq  \rds (\infty)$
\st\
\begin{enumerate}
\item[$(*)_1$] for every $c:[T]^{<\aleph_0} \rightarrow \mu$ there
is $T'\subseteq T$ isomorphic to $S$ \st\ $c\rest T'$ is
$c$-end-uniform. \item[$(*)_2$] $|T| < \beth_{|S|^+} (|S|+\mu)$.
\end{enumerate}
\end{mclaim}
\begin{proof} We assume that $|S|,\mu$ are infinite cardinals  since one of our main goals is proving a statement of the form $x\rightarrow[y]^n_{\mu,\aleph_0}$, otherwise the bound on $T$ has to be slightly adjusted.
\\For each $\eta\in S$ let
$$
\begin{array}{l}\alpha_\eta=\alpha_S (\eta)={\rm otp} (\{\nu\in S :\nu\ordtwo\eta\},\ordtwo),
\\ \mu_\eta=\beth_{5 \alpha_\eta+1} (|S|+\mu), \\ \lambda_\eta=\beth_3(\mu_\eta)^+. \end{array}$$
Note that $\mu_{\seq{}}, \lambda_{\seq{}}$ are the maximal ones, and let  $\chi >>\lambda_{<>}$, and $<^*_\chi$ be a well ordering of $\mat H(\chi)$ (see \ref{logics-def}(5)). By definition, for every $\eta,\nu\in S$ \st\ $\eta\ordtwo\nu$ we have $\mu_\eta<\mu_\nu$, and $\lambda_\eta<\lambda_\nu$ in the following we examine the relation between $\mu_\nu$ and $\lambda_\eta$ for $\eta\neq\nu$.
\begin{obs}
\label{1.11}
For $\eta\ordtwo\nu$ we have $\mu_\nu\geq\lambda_\eta^+$.
\end{obs}
\begin{proof} Since $\alpha_{\nu}\geq\alpha_\mu+1$ we have:
 \\$\begin{array}{lcl}
\mu_\nu & = &\beth_{5\alpha_\nu+1}(\size{S}+\mu)\\
\ & \geq & \beth_{5(\alpha_\eta+1)+1}(\size{S}+\mu)\\
\ & = &\beth_{5}(\mu_\eta)\\
\ & \geq & \beth_3(\mu_\eta)^{++}\\
\ & = & \lambda_\eta^+
\end{array}$ \\ 
\end{proof}
let $T:=\rds (\lambda_{\seq{}}^+)$, we will show that $T$ is as required. Obviously $T$ meets requirement $(*)_2$, and let $c:[T]^{<\aleph_0}\rightarrow \mu$. Because of the many details in the following construction we bring it as a separate lemma.
\begin{lemma}
\label{induction-prop}
 For $\eta\in S$ we can choose $M_\eta$, $T^*_\eta$ and $\nu_{\eta,n}\in T$ for $n<\omega$ with the following properties:
\begin{enumerate}
\item $M_\eta$ is an $\logic{\mu_\eta^+}$-elementary submodel of $\bool=({\mat H} (\chi),\in,<^*_\chi)$. 
 \item $\|M_\eta\|=2^{\mu_\eta}$.
 \item$S,T,c\in M_\eta$.
 \item $M_\rho,\nu_{\rho,n}\in M_\eta$ for  $\rho<_{\ell x}^* \eta$, $n<\omega$.
\item Properties of $T^*_\eta$:
 \begin{enumerate}
 \item $ T^*_\eta=\nu_{\eta,\rlg(\eta)}\conc T'$ where $T'$ is isomorphic to $\rds(2^{2^{\mu_\eta}})$.
 \item If $\nu',\nu''\in T^*_\eta$ and are of the same length then they realize the same $\logic{\mu_\eta^+}$-type over $M_\eta$.
     \end{enumerate}
\item Properties of the  $\nu_{\eta,n}$:
\begin{enumerate}
\item $\nu_{\eta,n}\in T$ is of length $n$.
\item $\nu_{\eta,\rlg (\eta)}\in M_\eta$.
\item $\rlg(\eta)= m<n\Rightarrow\nu_{\eta,n}(m)\notin M_\eta$.
\item $\nu_{\eta,n}\in T^*_\eta$, and for $n\geq\rlg(\eta)$ has at least $\mu_\eta$ immediate successors in $T^*_\eta$.
    \end{enumerate}
\item If $\eta=\eta_1 \conc \seq{\alpha}$, \then\
\begin{enumerate}
\item $M_\eta,T_\eta^*,\nu_{\eta,n} \in M_{\eta_1}$ for $n<\omega$.
\item $\nu_{\eta_1,n}, \nu_{\eta,n}$ realize the same ${\bb L}_{\mu^+_\eta,\mu^+_\eta}$-type over
    $\{M_\rho,\nu_{\rho,n}:n<\omega,\rho <^*_{\ell x} \eta\}$.
\item $\nu_{\eta_1,n}=\nu_{\eta,n}$ for $n\leq\rlg\ (\eta_1)$.
 \item $\nu_{\eta,n} \order\nu_{\eta_1,n}$ for $n=\rlg (\eta)$.
 \item $\nu_{\eta,\rlg(\eta)}=\nu_{\eta,\rlg\eta_1}\conc\langle\gamma\rangle$ for some $\gamma$.
     \item If $\eta'=\eta_1\conc\seq{\alpha'}$  with $\alpha'<\alpha$ then $\nu_{\eta',\rlg(\eta')}\order\nu_{\eta,\rlg(\eta)}$.
\end{enumerate}
\end{enumerate}
\end{lemma}
\begin{proof}
We show a construction for such a choice  by induction on $\ordone$, yes, $\ordone$ not $\ordtwo$.
\\ As the induction is on $\ordone$ the base of the induction is the case $\eta=\seq{}$.  First choose $M_{\seq{}} \prec_{\logic{\mu^+_{\seq{}}}}\bool$ of cardinality $2^{\mu_{\seq{}}}$, so that $S,T,c\in M_{\seq{}}$ (this can be done, see Remark \ref{logics-use}). The number of  $\bb L_{\mu_{\seq{}}^+,\mu_{\seq{}}^+}$ formulas $\varphi(\bar x,\bar a)$ where $\bar a\subseteq{}^{\mu_{\seq{}}^+>}M_{\seq{}}$ (sequences of length $< \mu_{\seq{}}^+$ in $M_{\seq{}}$) is $\leq (2^{\mu_{\seq{}}})^{\mu_{\seq{}}}=2^{\mu_{\seq{}}}$ hence the number of $\logic{\mu^+_{\seq{}}}$-types over $M_{\seq{}}$ is at most $\mu'=2^{2^{\mu_{\seq{}}}}$, so we colour $T=\rds (\lambda^+_{\seq{}})$ by $\leq \mu'$ colours, $c_{\seq{}}:T\to\mu'$, so that for $\rho\in T$ its colour, $c_{\seq{}}(\rho)$, codes the $\logic{\mu^+_{\seq{}}}$-type which $\rho$ realizes in $\bool$ over $M_{\seq{}}$. As  $$((\beth_2(\mu_{\seq{}}))^{\mu'^{\aleph_0}})^+=\beth_3(\mu_{{\seq{}}})^+=\lambda_{{\seq{}}}$$  by Theorem \ref{796} there is an embedding of $\rds(\beth_2(\mu_{\seq{}}))$ in $T$, and define $T^*_{\seq{}}$ to be its image, so that types of \sequ s from $T^*_{\seq{}}$ depend only on their length. We choose representatives
$\seq{ \nu_{{\seq{}},n}:0<n<\omega}$  from each level larger than $0$ so that for $n>0$ $\nu_{{\seq{}},n}$ and has at least $\mu_{\seq{}}$ immediate successors in $T^*_{\seq{}}$ and satisfies $6(c)$. The latter  can be done by cardinality considerations, $\|M_{\seq{}}\|=2^{\mu_{\seq{}}},$ while the cardinality of levels in $T^*_{\eta_{\seq{}}}$ is $\beth_2(\mu_{\seq{}})$.
We let $\nu_{\seq{},0}=\seq{}$.
\\It is easily verified that for $\eta=\seq{}$ all the requirements of the construction are met.
\\We now show the induction step.  \\Assume $\eta=\eta_1\conc\seq{\alpha_1}$, $\rlg(\eta_1)=r$, and that we have defined for
$\eta_1$ (and below by $\ordone$) and we define for $\eta$.
\begin{enumerate}
  \item[$\circledast_1$] Let $A_\eta=\{M_\rho,\nu_{\rho,n}: n<\omega,\rho <^*_{\ell x} \eta\}.$
\end{enumerate}
For any $\rho<^*_{\ell x}\eta$ if $\rho=\eta_1\conc\langle\alpha\rangle$ for some $\alpha<\alpha_1$ then from requirement $(7)(a)$ of the construction for $\rho$ we have $M_\rho\in M_{\eta_1}$, and also for all $n<\omega$ $\nu_{\rho,n}\in M_{\eta_1}$, else $\rho\order\eta_1$ therefore from requirement $(4)$ of the construction for $\eta_1$ we have for all $n<\omega$ $\nu_{\rho,n}\in M_{\eta_1}$, and $M_\rho\in M_{\eta_1}$. So $A_\eta\subseteq M_{\eta_1}$, and $\size{A_\eta}\leq\mu_{\eta_1}$, so $A_\eta$ is definable by an $\logic{\mu_{\eta_1}^+}$-formula with parameters in $M_{\eta_1}$, so we have:
\begin{enumerate}
    \item[$\circledast_2$]  $A_\eta \subseteq M_{\eta_1}, |A_\eta|\leq\mu_\eta\leq \mu_{\eta_1}$, therefore $A_\eta\in M_{\eta_1}$.
\end{enumerate}
For every $n<\omega$ let
\begin{enumerate}
    \item[$\circledast_3$] $\varphi_{n}(x)=\varphi_{\mu_{\eta_1},n}(x)= \bigwedge (\hbox{ the } \logic{\mu^+_{\eta}}-\hbox{ type which } \nu_{\eta_1,n} \hbox{ realizes over } A_\eta)$
\end{enumerate} And  let
\begin{enumerate}
    \item[$\circledast_4$]$T_{\varphi}=\{\rho\in T: \bool\models \varphi_{\rlg(\rho)} (\rho)\}.$
\end{enumerate}
As the cardinality of the $\logic{\mu^+_\eta}$-type of any $\nu\in\bool$ over $A_\eta$ is at most $2^{\mu_\eta}$ which is less than $\mu_{\eta_1}$, for every $n<\omega$ we have that $\varphi_n$ is an $\logic{\mu^+_{\eta_1}}$-formula  and therefore $T_{\varphi}$ is definable in $M_{\mu_{\eta_1}}$ by an $\logic{\mu^+_{\eta_1}}$-formula, namely  $$
    \rho\in T_{\varphi}\leftrightarrow\Big(\rho\in T\wedge \big(\bigvee\limits_{n<\omega}({\rm lg}(\rho)=n\wedge\varphi_n(\rho))\big)\Big)$$
So \begin{enumerate}
     \item[$\circledast_5$]$T_{\varphi} \in M_{\eta_1}$ and for every $n<\omega$ we obviously  have $\nu_{\eta_1,n}\in T_{\varphi}$.
\end{enumerate}
Recall that for all $n<\omega$ $\nu_{\eta_1,n}\in T^*_{\eta_1}$,
 so for any $\rho\in T^*_{\eta_1}$ of length $n$, we have that $\rho$
realizes the same $\logic{\mu_{\eta_1}^+}$-type over $M_{\eta_1}$
as $\nu_{\eta_1,n}$ so in particular they realize the same
$\logic{\mu_\eta^+}$-type over $A_\eta$, so $\rho\in T_\varphi$.
For $m\geq n$ $\nu_{\eta_1,n},\nu_{\eta_1,m}\rest n$ are of the
same length, so in particular $\varphi_m(x)\vdash\varphi_n(x\rest
n)$. If $\rho\in T_{\varphi}$, $\rlg \rho=m $  so
$\bool\models\varphi_{m}(\rho)$ therefore
$\bool\models\varphi_n(\rho\rest n)$ and therefore also $\rho\rest
n\in T_{\varphi}$. We summarize:
 \begin{enumerate}
   \item[$\circledast_6$]  $T_{\varphi}$  is a
    subtree of \ $T$ and $T^*_{\eta_1}\subseteq T_\varphi$.
 \end{enumerate}
The following point is a crucial one, we show that:
   \begin{enumerate}
   \item[$\circledast_7$]  ${\rm rk}_{T_{\varphi},\mu_{\eta_1}}
    (\nu_{\eta_1,n}) > \mu_{\eta_1}$ for every $n$ \st\  $\rlg(\eta_1)\leq n<\omega$ .
    \end{enumerate}
Assume toward contradiction that ${\rrk}_ {T_{\varphi},\mu_{\eta_1}}
(\nu_{\eta,m}) \leq \mu_{\eta_1}$ for some $\rlg(\eta_1)\leq
m<\omega$, and define for each $n$ \st\ $m\leq n<\omega:$ $$\gamma_n=
\rrk_{T_{\varphi},\mu_{\eta_1}}(\nu_{\eta,n})\hbox{ and }
\gamma^*_n=\rrk^{\mu_{\eta_1}}_{T_\varphi,\mu_{\eta_1}}(\nu_{\eta_1,n})$$
(see Definitions \ref{rank} and \ref{reduced-rank}). We now prove by
induction on $n\geq m$ that $\gamma_{n+1}\leq\mu_{\eta_1}$, i.e
$\gamma_n=\gamma^*_n$. For $n=m$ this is our assumption, and assume
that it is known for $n$.
   The following  can be expressed by $\logic{\mu_{\eta_1}^+}$-formulas with parameters in $M_{\eta_1}:$
   \begin{enumerate}
   \item[$\psi_1:$]`$x$ has $\rrk_{T_\varphi,\mu_{\eta_1}}^{\mu_{\eta_1}}(x)=\gamma_n$'
   \item[$\psi_2:$] `$x$ has at least $\mu_{\eta_1}$ immediate successors $y$ in $T_\varphi$ with $\rrk_{T_{\varphi},\mu_{\eta_1}}^{\mu_{\eta_1}}(y)\geq \gamma^*_{n+1}$'
   \end{enumerate}
We have  $\bool\models\psi_1(\nu_{\eta_1,n})$, and since   $T_{\eta_1}^*\subset T_\varphi$ (see $\circledast_6$) we also have $\bool\models\psi_2(\nu_{\eta_1,n})$. By the induction hypothesis for $\eta_1$ we have $\nu_{\eta_1,n},\nu_{\eta_1,n+1}\rest n\in T^*_{\eta_1}$ and as they are the same length realize the same $\logic{\mu_{\eta_1}^+}$-type over $M_{\eta_1}$, so $\bool\models\psi_1\wedge\psi_2(\nu_{\eta_1,n+1}\rest n)$, or in more detail, we  have  that $\rrk_{T_\varphi,\mu_{\eta_1}}^{\mu_{\eta_1}}(\nu_{\eta_1,n+1}\rest n)=\gamma_n$, i.e  $\rrk_{T_\varphi,\mu_{\eta_1}}(\nu_{\eta_1,n+1}\rest n)=\gamma_n$, and  $\nu_{\eta_1,n+1}\rest n$ has at least $\mu_{\eta_1}$ immediate successors in $T_\varphi$ with reduced rank $\gamma^*_{n+1}$, so by the definition of rank (Definition \ref{rank}) we have $\gamma_n>\gamma^*_{n+1}$. By the induction hypothesis $\gamma_n\leq\mu_{\eta_1}$, therefore also  $\gamma_{n+1}^*=\gamma_{n+1}$. In particular we can deduce that  $\gamma_{n+1}<\gamma_n$, so having carried out the induction we have an infinite decreasing sequence of ordinals which is a contradiction.
\\Recall that $\rlg(\eta_1)=r$ so $\rlg(\eta)=r+1$,
\begin{enumerate}
\item[$\circledast_{8}$] Define $\nu_{\eta,\ell}=\nu_{\eta_1,\ell}$ for $\ell\leq r$.
\end{enumerate}
 By \ref{1.11}  $\mu_{\eta_1}\geq\lambda_\eta^+$,
 by $\circledast_{7}$ $\rrk_{T_{\varphi},\mu_{\eta_1}}(\nu_{\eta_1,r})>\mu_{\eta_1}$ therefore $\rrk_{T_{\varphi},\mu_{\eta_1}}(\nu_{\eta_1,r})>\lambda_\eta^+$ so by definition
 there are $\nu\in {\rm Suc}_T (\nu_{\eta_1,r}) \cap
 T_{\varphi}$ satisfying ${\rrk}_{T_{\varphi},\mu_{\eta_1}}
    (\nu)\geq \lambda_\eta^+$, defining $\nu_{\eta,r+1}$ to be one such $\nu$  which is minimal with respect to $\ordone$ (this is equivalent to demanding that $\nu(r)$ is minimal) can be done by an $\logic{\mu_{\eta_1}^+}$ formula. We therefore conclude:
    \begin{enumerate}
\item[$\circledast_{9}$] We can  choose $\nu_{\eta,r+1} \in {\rm Suc}_T (\nu_{\eta_1,r}) \cap T_{\varphi}\cap M_{\eta_1}$ \st\
    \begin{enumerate}
    \item[$(i)$]
    ${\rm rk}_{T_{\varphi},\mu_{\eta_1}}
    (\nu_{\eta,r+1}) \geq \lambda_\eta^+$.
    \item[$(ii)$]$\nu_{\eta,r+1}$ is minimal under $(i)$  in $\ordone$.
    \end{enumerate}
\end{enumerate}
As $\nu_{\eta,\rlg(\eta)}\in M_{\eta_1} and $ $\nu_{\eta_{1},\rlg(\eta)}(\rlg(\eta_1))\notin M_{\eta_1}$, we have:
\begin{enumerate}
\item[$\circledast_{10}$]$\nu_{\eta,\rlg(\eta)}\ordone \nu_{\eta_1,\rlg(\eta)}$,  notice that as they are the same length $\ordone\Rightarrow\order$.
\end{enumerate}
Now for any $\rho=\eta_1\conc\seq{\alpha}\in S$ where $\alpha<\alpha_1$ we have that $\rho\order\eta$ and therefore $\nu_{\rho,r+1}\in A_{\eta}$ (see $\circledast_1$). $\nu_{\eta,\rlg(\eta)},\nu_{\eta_1,\rlg(\eta)}$ realize the same $\logic{\mu_\eta^+}$-type over $A_\eta$, and by requirement $(7)(d)$ of the construction for $\rho$ ($\rlg(\rho)=\rlg(\eta)$) we have $\nu_{\rho,\rlg(\eta)}\ordone\nu_{\eta_1,\rlg(\eta)}$ so also  $\nu_{\rho,\rlg(\eta)}\ordone\nu_{\eta,\rlg(\eta)}$ and as above, as they are the same length $\ordone\Rightarrow\order$, and we therefore conclude that:
\begin{enumerate}
    \item[$\circledast_{11}$] If $\rho=\eta_1\conc\seq{\alpha}\in S$ where $\alpha<\alpha_1$
    then $\nu_{\rho,\rlg(\eta)}\order\nu_{\eta,\rlg(\eta)}$.
\end{enumerate}
Since $\size{\set{S,t,c,\nu_{\eta_\rlg(\eta)}}\cup A_\eta}<2^{\mu_\eta}$ by Remark \ref{logics-use} we can   choose $M_\eta$ so that
\begin{enumerate}
\item[$\circledast_{12}$] $M_\eta\prec_{\logic{\mu_\eta^+}}M_{\eta_1}$, and therefore also $M_\eta\prec_{\logic{\mu_\eta^+}}\bool$, of cardinality $2^{\mu_\eta}$ and $\set{S,t,c,\nu_{\eta_\rlg(\eta)}}\cup A_\eta\subseteq M_\eta$.
\end{enumerate}
By the same remark we can conclude that
\begin{enumerate}
\item[$\circledast_{13}$] $M_\eta\in M_{\eta_1}$.
\end{enumerate}
Lastly we choose $T^*_\eta$ and $\nu_{\eta,m}$ for $m>\rlg (\eta)$.
\\We have already commented that $\rrk_{T_\varphi,\mu_{\eta_1}}(\nu_{\eta,\rlg(\eta)})>\lambda_\eta^+$, so from Observation \ref{obs-rank} we can embed $\nu_{\eta,\rlg(\eta)}\conc\rds(\lambda_\eta^+)$ into $T_{\varphi}$ so that $\rho\mapsto\rho$ for $\rho\trianglelefteq\nu_{\eta,\rlg(\eta)}$, and denote one such embedding by $\psi$, \wolog\ $\psi\in M_{\eta_1}$.
\\The number of $\bb L_{\mu_\eta^+,\mu_\eta^+}$-types over $M_\eta$ is at most $\mu'=2^{2^{\mu_\eta}}$. We colour $\rds(\lambda_\eta^+)$ in $\leq\mu'$ colours, the colour of $\rho\in\rds(\lambda_\eta^+)$ is determined by the $\bb L_{\mu_\eta^+,\mu_\eta^+}$-type which $\psi(\nu_{\eta,\rlg(\eta)}\conc\rho)$ realizes over $M_\eta$, call this colouring $c_\eta$.
As $((\beth_2(\mu_\eta))^{\mu'^{\aleph_0}})^+=\beth_3(\mu_\eta)^+=\lambda_\eta$, we can use \ref{796} to get an embedding $\theta$ of $\rds(\beth_2(\mu_\eta))$ into $\rds(\lambda_\eta^+)$ so that for $\rho\in\rds(\beth_2(\mu_\eta))$ the $\logic{\mu_\eta^+}$-type that $\nu_{\eta,n+1}\conc\theta(\rho)$ realizes over $M_\eta$ depends only on its length. Since the set $X$ of $\logic{\mu_\eta^+}$-types over $M_\eta$ is in $M_{\eta_1}$ of cardinality at most $\mu'<\mu_{\eta_1}$ we have $X\subset M_{\eta_1}$,  also $\rds(\lambda_\eta^+)\in M_{\eta_1}$ so $c_\eta\in M_{\eta_1}$ and therefore \wolog\ $\theta\in M_{\eta_1}$.   We define
\begin{enumerate}
\item[$\circledast_{14}$] $T^*_{\eta}=\nu_{\eta,\rlg(\eta)}\conc\theta\big(\rds(\beth_2(\mu_\eta))\big)$.
\end{enumerate}
$T^*_\eta\in M_{\eta_1}$ and  meets requirement $(5)$ of the construction.
We will now choose representatives $\seq{\rho_m:0<m<\omega}$ from each level of $\rds(\beth_2(\mu_\eta))$ so that $\nu_{\eta,n+1}\conc\theta(\rho_m)$ has at least ${\mu_{\eta}}$ immediate successors in $T^*_\eta$ and $\nu_{\eta,n+1}\conc\theta(\rho_m)(\rlg(\eta))\notin M_{\eta_1}$, since the existence of such representatives in $\bool$ can be expressed by an $\logic{\mu_{\eta_1}^+}$-formula with parameters in $M_{\eta_1}$  so \wolog\ $\rho_m\in M_{\eta_1}$ and define
\begin{enumerate}
\item[$\circledast_{15}$] $\nu_{\eta, \rlg(\eta)+m}=\nu_{\eta,n+1}\conc\theta(\rho_m)$.
\end{enumerate}
$T^*_\eta$ is a subtree of $T_\varphi$ therefore $\rho\in T^*_\eta$ realizes the same $\logic{\mu_\eta^+}$type over $A_\eta$ as $\nu_{\eta_1,\rlg(\rho)}$.
The $\nu_{\eta,n}$ for $n>\rlg(\eta)$ were chosen to satisfy $(6)(c)$-$(d)$ so in particular they are in $T_\varphi$, and therefore realize the same $\logic{\mu_\eta^+}$-type over $A_\eta$ as $\nu_{\eta_1,n}$. By the induction hypothesis we have already constructed for $\eta_1$ so  for all $n$ we have  $\rlg(\nu_{\eta,n})=\rlg(\nu_{\eta_1,n})=n$ so also $(6)(a)$ is satisfied.
Requirements $(1)$-$(4)$ and $(6)(b)$ of the construction  are taken care of by $\circledast_{12}$.
$\circledast_7$-$\circledast_{11}$, $\circledast_{13}$ and $\circledast_{15}$ guarantee requirement $(7)$.
\end{proof}
\par\noindent
\\All that is left in order to complete the proof of  the claim is to show that
 $\{\nu_{\eta,\rlg (\eta)}:\eta\in S\}$ is end-uniform with respect to $c$.
\\Let $\eta_1\ordtwo \eta_2\ordtwo \ldots \ordtwo\eta_k\ordtwo\rho', \rho''$,
be as in \ref{uniformity}(2); \wolog\ $\rho' \order\rho''$. Let
$t=\rlg (\rho'\cap \rho'')$, $\mu'=\mu^+_{\rho'}$ and
$A=\set{\nu_{\rho,\rlg \rho}: \rho \order \rho'\rest (t+1)}$.
\\We first show that for every $i\leq k$ $\eta_i\order\rho'\rest(t+1)$
so that $\nu_{\eta_i.\rm lg(\eta_i)}\in A$. As $\eta_i\ordtwo\rho'$
and $\rlg(\eta_i\cap\rho'')=\rlg(\eta_i\cap\rho')$ so
$\rho'\ntriangleleft\eta_i$, therefore there is $\ell_i$ \st\
$\eta_i\rest\ell_i=\rho'\rest\ell_i$ and
$\eta_i(\ell_i)<\rho'(\ell_i)$, but then
$\eta_i\rest\ell_i=\rho''\rest\ell_i$ i.e
$\rho'\rest\ell_i=\rho''\rest\ell_i$ so $\ell_i\leq t$ (and
$\eta_i(\ell_i)<\rho''(\ell_i)$) and $\eta_i\order\rho'\rest(t+1)$.
\\We now prove by induction on $\ell\in [t,\rlg (\rho')]$ that $\nu_{\rho'\rest\ell,\rlg \rho'}$
and $\nu_{\rho'\rest t,\rlg \rho'}$ realize the same ${\bb
L}_{\mu',\mu'}$-type over $A$. For $\ell=t$ this is obvious. Let us
assume correctness  for $\ell$ and prove for $\ell + 1$. For every
$n<\omega$ by $(7)(b)$ of the construction
$\nu_{\rho'\rest\ell,n},\nu_{\rho'\rest(\ell+1),n}$ realize the same
$\logic{\mu^+_{\rho'\rest(\ell+1)}}$-type over
$\set{M_\rho,\nu_{\rho,n}:\rho\order\rho'\rest(\ell+1)}$ and in
particular over $A$, for if $\rho\order\rho'\rest(t+1)$ then also
$\rho\order\rho'\rest(\ell+1)$. So
$\nu_{\rho'\rest\ell,\rlg\rho'},\nu_{\rho'\rest(\ell+1),\rlg\rho}$
realize the same ${\bb L}_{\mu^+_{\rho'\rest(\ell+1)},
\mu^+_{\rho'\rest(\ell+1)}}$-type so also the same ${\bb
L}_{\mu',\mu'}$-type over $A$, and from the induction hypothesis
$\nu_{\rho'\rest t,\rlg\rho'}$ and $\nu_{\rho'\rest\ell, \rlg\rho'}$
realize the same $\logic{\mu'}$-type over $A$.
    Similarly we show for $\rho''$, so $\nu_{\rho',\rlg \rho'}$ and $\nu_{\rho'',\rlg \rho''}$ realize the same ${\bb L}_{\mu^+_{\eta_1},\mu^+_{\eta_1}}$-type over $A$.
    \newline From the above we can deduce that in particular  $$c(\seq{\nu_{\eta_1,\rlg(\eta_1)},\dots,\nu_{\eta_k,\rlg(\eta_k)},\nu_{\rho',\rlg(\rho')}})=
    c(\seq{\nu_{\eta_1,\rlg(\eta_1)},\dots,\nu_{\eta_k,\rlg(\eta_k)},\nu_{\rho'',\rlg(\rho'')}}).$$
\end{proof}

\begin{conclusion}
\label{conclusion}
 Given a tree $S \subseteq \rds (\infty)$
and $n(*)<\omega$ and $\mu$ we can find a tree
$T\subseteq \rds (\infty)$ \st:
\begin{enumerate}
\item [$(\ast)_1$] For every $c:[T]^{<\aleph_0}
\rightarrow \mu$ there is $S' \subseteq T$
isomorphic to $S$ \st\ $S'$ is $n(*)$-end-uniform for
$c$.
\item [$(\ast)_2$]In particular, for every
    $c:[T]^{n(*)}\rightarrow\mu$ is $S'\subseteq T$ isomorphic to
    $S$ \st\ $c\rest S'$ depends only on the equivalence classes
    of the equivalence relation defined in \ref{similarity}.
\item [$(\ast)_3$]$|T|<\beth_{1,n(*)}(|S|,\mu)$ (see Definition \ref{beths} below).
\end{enumerate}
\end{conclusion}
\begin{proof} Let $S$, $\mu$ be as above. Since for $\size{S}, \mu\geq\aleph_0$ we have that  $\beth_{1,n(*)}(|S|,\mu^{\aleph_0})= \beth_{1,n(*)}(|S|,\mu)$, replacing $\mu$ with $\mu^{\aleph_0}$ gives the same bound, and we can therefore assume that $\mu=\mu^{\aleph_0}$.
\\Let $\seq{h_n:n<\omega}$ be the equivalence classes of the similarity relationship on finite sequences of $\rm ds(\infty)$
(see \ref{obs-sim}(1)), and let
$f:{}^\omega(\mu\cup\{-1\})\rightarrow\mu$ be one-to-one and onto.
\\We construct by induction a \sequ\ $\seq{T_n:n<\omega}$ so that $T_0=S$, and for every $n>0$:
\begin{enumerate}
\item[$(a)$]  $|T_n|<\beth_{1,n}(|S|,\mu)$
\item[$(b)$] $T_{n-1},T_n,\mu$ correspond to $S,T,\mu$ in Theorem \ref{main claim}.
\item [$(c)$] For every $c:[T_n]^{<\aleph_0}\rightarrow \mu$  there is $S' \subseteq T_n$
isomorphic to $S$ \st\ $S'$ is $n$-end-uniform for $c$.
\end{enumerate}
By Theorem \ref{main claim} we can obviously construct such a sequence satisfying clauses $(a),(b)$, We will show by induction on $n$ that for this sequence also clause $(c)$ holds.  For $n=1$ this is Theorem \ref{main claim}. Assume correctness for $n$ and let $c:[T_{n+1}]^{<\aleph_0}\rightarrow \mu$. By $(b)$ there is $T' \subseteq T_{n+1}$ isomorphic to $T_n$ so that  $T'$ is end-uniform for $c$. Let $\varphi:T_n\rightarrow T'$ be an isomorphism and let $d:[T']^{<\aleph_0}\rightarrow{}^\omega(\mu\cup\{-1\})$ as follows: for $\bar\rho=\seq{\rho_1\dots\rho_k}$ where $\rho_1\ordtwo\rho_2\ordtwo\dots\ordtwo\rho_k$ and $m<\omega$
$$
d(\bar\rho)(m)=\left\{\begin{array}{ll}
 $$c(\bar\rho\conc\seq{\eta})$$ & \textrm{if $\bar\rho\conc\seq{\eta}\in h_m$ for some $\eta$}\\
$-1$ & \textrm{otherwise}
\end{array} \right. $$\\ $d$ is well defined as $T'$ is end-uniform for $c$,  and by defining $\varphi(\rho_1,\dots\rho_k)=(\varphi(\rho_1),\dots\varphi(\rho_k))$ for $\rho_1,\dots\rho_k\in T_n$ we have $f\circ d\circ\varphi:[T_n]^{<\aleph_0}\rightarrow\mu$, so by the induction hypothesis there is $T''\subseteq T_n$ isomorphic to $S$ so that $T''$ is $n$-end-uniform for $f\circ d\circ\varphi$. We claim that $S'=\varphi(T'')$ is isomorphic to $S$ and that $S'$ is $n+1$-end-uniform for $c$.  As $T''$ is isomorphic to $S$ and $\varphi$ is an isomorphism $S'$ is obviously isomorphic to $S$. Let the following sequences in $S'$ be similar,
$$\eta_1 \ordtwo<\eta_2 \ordtwo\ldots \ordtwo \eta_k
\ordtwo \rho'_1\ordtwo\ldots \ordtwo
\rho'_{n+1}
$$
$$
\eta_1 \ordtwo < \eta_2 \ordtwo\ldots
\ordtwo \eta_k \ordtwo
\rho''_1 \ordtwo <\ldots< \rho''_{n+1}
$$
So in $T''$ the following sequences are similar:
$$
\varphi^{-1}(\eta_1  \dots  \rho'_1\dots\rho'_n)=(\varphi^{-1}(\eta_1 )\varphi^{-1}( \rho'_1)\dots\varphi^{-1}(\rho'_n))
$$
$$
\varphi^{-1}(\eta_1  \dots  \rho''_1\dots\rho''_n)=(\varphi^{-1}(\eta_1 )\varphi^{-1}( \rho''_1)\dots\varphi^{-1}(\rho''_n))
$$
so $f\circ d\circ\varphi(\varphi^{-1}(\eta_1  \dots \eta_k, \rho'_1\dots\rho'_n))=f\circ d\circ\varphi(\varphi^{-1}(\eta_1  \dots \eta_k, \rho''_1\dots\rho''_n))$. Therefore we have $f(d(\eta_1  \dots \eta_k, \rho'_1\dots\rho'_n))=f(d(\eta_1  \dots \eta_k, \rho''_1\dots\rho''_n))$, and as $f$ is one-to-one, $d(\eta_1  \dots \eta_k, \rho'_1\dots\rho'_n)=d(\eta_1  \dots \eta_k, \rho''_1\dots\rho''_n)$, and therefore  $c(\eta_1  \dots \eta_k, \rho'_1\dots\rho'_{n+1})=c(\eta_1  \dots \eta_k, \rho''_1\dots\rho''_{n+1})$, and  $(\ast)_1$-$(\ast)_3$ are easily verified.
\end{proof}

\begin{definition}
\label{beths}
For cardinals $\lambda\geq\aleph_0$ and $\mu$ define $\beth_{1,\alpha}(\lambda,\mu)$ by induction on $\alpha$. $\beth_{1,0}(\lambda,\mu)=\beth_0(\lambda)=\lambda$, $\beth_{1,\alpha+1}(\lambda,\mu)=\beth_{\beth_{1,\alpha}(\lambda,\mu)^+}(\beth_{1,\alpha}(\lambda,\mu)+\mu)$, and for a limit ordinal $\alpha$ $\beth_{1,\alpha}(\lambda,\mu)=\sum\limits_{\beta<\alpha}\beth_{1,\beta}(\lambda,\mu)$.
\end{definition}

We end with a conclusion for scattered order types.
\begin{conclusion}
For a scattered order type $\varphi$, a cardinal $\mu$ and $n<\omega$, there is a scattered order type $\psi$ so that $\psi\rightarrow(\varphi)^{n}_{\mu,\aleph_0}$.
\end{conclusion}
\begin{proof} Given a scattered order type $\varphi$, a cardinal $\mu$ and $n<\omega$  by Observation  $\ref{obs-ord}(3)$ we can embed $\varphi$ in $(\rds(\alpha),<^3)$ for some ordinal $\alpha$. By Conclusion $\ref{conclusion}(*)_2$ above there is an ordinal $\lambda$ and a tree $T\subset\rds(\lambda)$ so that for every colouring $c:T^n\rightarrow\mu$ there is a subtree $S\subseteq T$ isomorphic to $\rds(\alpha)$ so that $c\rest S$ depends only on the equivalence class of similarity. Noting the above Observation, as $(T, <^3)$ is a scattered order, and as there are only $\aleph_0$ equivalence classes, we are done.
\end{proof}

\providecommand{\bysame}{\leavevmode\hbox to3em{\hrulefill}\thinspace}
\providecommand{\MR}{\relax\ifhmode\unskip\space\fi MR }
\providecommand{\MRhref}[2]{%
  \href{http://www.ams.org/mathscinet-getitem?mr=#1}{#2}
}
\providecommand{\href}[2]{#2}

\end{document}